\documentclass[12pt]{article}

\title{On the joint behaviour of speed and entropy of random walks on groups.}
\date{\today}
\author{Gideon Amir}

\usepackage{amsmath,amsthm,amssymb,amsfonts,
enumerate,graphicx}

\newcounter{quest}
\theoremstyle{plain}
    \newtheorem{theorem}{Theorem}
    
    \newtheorem{lemma}[theorem]{Lemma}
    \newtheorem{proposition}[theorem]{Proposition}
    \newtheorem{corollary}[theorem]{Corollary}
    
    \newtheorem{conjecture}[quest]{Conjecture}
    \newtheorem{problem}[quest]{Problem}
    \newtheorem{question}[quest]{Question}

\theoremstyle{definition} 

\newcommand{\Z}{{\mathbb Z}}

\newcommand{\pr}{{\mathbf P}}
\newcommand{\R}{{\mathbb R}}

\newcommand{\tree }{\mathbb T}

\newcommand{\Aut}{\operatorname{Aut}}

\renewcommand{\bar}{\overline}

\newcommand{\ev}{\mbox{\bf E}}

\newcommand{\mm}{{\mathcal M_{m}}}

\newcommand{\tm}{{\tree_{m}}}
\newcommand{\ttm}{{\tree_{\tau m}}}

\begin{document}
\maketitle
\begin{abstract}
For every $3/4\le \delta, \beta< 1$ satisfying $\delta\leq \beta < \frac{1+\delta}{2}$ we construct a finitely generated group $\Gamma$ and a (symmetric, finitely supported) random walk $X_n$ on $\Gamma$ so that its expected distance from its starting point satisfies  $\ev |X_n|\asymp n^{\beta}$ and its entropy satisfies $H(X_n)\asymp n^\delta$. In fact, the speed and entropy can be set precisely to equal any two nice enough prescribed functions $f,h$ up to a constant factor as long as the functions satisfy the relation $n^{\frac{3}{4}}\leq h(n)\leq f(n)\leq \sqrt{{nh(n)}/{\log (n+1)}}\leq n^\gamma$ for some $\gamma<1$.
\end{abstract}

\section{Introduction.}\label{s:intro}
Let $G$ be an infinite, finitely generate group with a finite symmetric generating set $S$, and let $X_n$ be a finitely supported, symmetric random walk on $G$ for which the support of $X_1$ generates $G$.
We study the following two quantities:
\begin{enumerate}
\item The \textbf{rate of escape} of the random walk from its starting position $\ev|X_n|$. (Where $|\cdot|$ denotes the word norm w.r.t. $S$.)
\item The \textbf{entropy} of the random walk $H(X_n)=-\sum_{g\in G}P(X_n=g)\log P(X_n=g)$. (Where $0\log 0 =0$.)
\end{enumerate}
Starting with the works of Kesten \cite{Kes}, the rate of escape and the entropy have been connected to many other properties of the group, including the spectral radius, return probabilities, isoperimetric properties and volume growth(see e.g.\cite{Kes,HSC,SCZ}), to the Liouville property (\cite{Der,KV}), and to embeddings into Hilbert space (\cite{ANP,NP}). This list is far from complete and we refer the reader to  \cite{AV,Bri2} and \cite{Gou} for further background.

While for some classes of groups the behaviour of the two quantities are understood, it is in many cases hard to get good bounds on the entropy and rate of escape (e.g. for random walks on Grigorchuk's group).
It is an open question whether the asymptotic behaviours of the rate of escape and the entropy are an invariant of the group $G$, or may depend on the choice of the random walk. (And in the latter case, by how much they can vary.)

The main question we wish to address is
\begin{question}
 What is the possible joint behaviour of the rate of escape and the entropy for random walks on groups?
 \end{question}

We will be interested in classifying the behaviour up to constants. We write
$f\preceq g$ if $f(n)\leq cg(n)$ for some constant $c$ independent of $n$, and $f\asymp g$ if $g\preceq f \preceq g$.
We will therefore study the following variant of the above question:
\begin{question}  For which functions $f$ and $h$ can one construct a random walk $X_n$ on a group $G$ for which $\ev|X_n|\asymp f(n)$ and $H(X_n)\asymp h(n)$?
\end{question}

The question of studying the possible behaviours of each of the two quantities separately has been a long standing open problem since first raised by Vershik.
There has been significant progress in recent years, in particular with regards to the rate of escape. Let us quickly summarize some of the main results in this direction.
It is straightforward that $|X_n|\leq n$ and $H(X_n)\leq nH(X_1)$. It has been shown in \cite{LP} that for any random walk on a group $\ev|X_n|\succeq \sqrt{n}$, and Virag conjectured that for any random walk on a group of exponential growth $H(X_n)\succeq \sqrt{n}$.
Examples where $\ev|X_n|\asymp \sqrt{n}$ (diffusive behaviour) are easy to come by, and include random walks on nilpotent groups as well as many other examples. In many of these cases entropy is also quite well understood. For non-amenable groups, and in fact for any non-Liouville random walk we have ballistic behaviour -  $\ev|X_n|\asymp H(X_n)\asymp n$ (\cite{KV}).
It took a long time until new examples (neither diffusive nor ballistic) for behaviour of the rate of escape were found.
Erschler \cite{ErschlerDrift} showed that the canonical random walk on the $k-1$-times iterated wreath product of $\mathbb Z$ satisfies
$\ev |X_n|\asymp n^{1-2^{-k}}.$ Erschler also showed that it can oscillate between two different powers of n \cite{E4}, and that it can be arbitrarily close to n \cite{E3}.
Recently Brieussel \cite{Bri2} showed that for every $a\in [\frac{1}{2},1]$ there exists a random walk $X_n$ on a group $G$ for which $\limsup \frac{\log \ev|X_n|}{\log n} = a$.

The most relevant works, which can be seen as the starting point of this paper, are the following two results which gave ways to construct random walks on groups with prescribed entropy or rate of escape behaviour within a given range:
\begin{enumerate}
\item It has been shown in \cite{Bri1} and refined in \cite{AV} that for any "nice''  function $h$ between $n^{\frac{1}{2}}$ and $n$ there is a group $G$ and a random walk $X_n$ on $G$ so that $H(X_n)\asymp h(n)$ ( "nice'' in the sense on Theorem \ref{t:main speed entropy} below).
\item It has been shown in \cite{AV} that for any $\gamma<1$ and any "nice'' function $f$ between $n^{\frac{3}{4}}$ and $n^\gamma$ there is a group $G$ and a random walk $X_n$ on $G$ satisfying $\ev|X_n|\asymp f(n)$.
\end{enumerate}

However, in both of the constructions mentioned above, one could only make either the rate of escape or the entropy behave in a prescribed way, and not both.
To be more precise, in both the above constructions, once we fixed one quantity (say the entropy) to behave like some pre-given function, we would get a specific random walk on a group, and that would set the second quantity. When fixing the rate of escape to be $\asymp f(n)$ one got that
$ H(X_n)\asymp\frac{f^2(n)}{n}\log n$ (this follows from bounds in section \ref{s:bounds}), and when fixing the entropy to be $\asymp h(n)$ one could only derive that $h(n)\preceq \ev|X_n| \preceq \sqrt{nh(n)}$.
So it is natural to ask whether one can control both the rate of escape and the entropy simultaneously. That is which pairs of speed and entropy functions are (explicitly?) attainable for random walks on groups. One well-known constraint is that
the rate of escape and the entropy satisfy the "speed-entropy relation'' (see e.g. \cite{ErsSE})
$$ H(X_n) \preceq \ev|X_n| \preceq \sqrt{n(H(X_n)+1)}. $$

The main theorem in this paper shows that for nice enough functions in the range above $n^\frac{3}{4}$ this is essentially the only constraint.
\begin{theorem}\label{t:main speed entropy}
For any $\gamma\in[3/4,1)$ and any functions $f,h:\R_+\rightarrow \R_+$ satisfying
$f(1)=h(1)=1$, the log-Lipshitz condition that for all real $a,n\geq 1$
$$a^{3/4}f(n)\leq f(an)\leq a^\gamma f(n) ,\,\,  a^{3/4}h(n)\leq h(an)\leq a^\gamma h(n)$$ and the relation $$h(n)\preceq f(n)\preceq \sqrt{\frac{nh(n)}{\log (n+1)}}$$
 there is
a group $\Gamma$ and a random walk $X_n$ on $\Gamma$ for which
$$
\ev|X_n|  \asymp f(n) , H(X_n)\asymp h(n).
$$
Where $\asymp (\preceq)$ denotes equality (inequality) up to constants depending on $\gamma$ only.
\end{theorem}

The groups and random walks involved are explicitly stated in the proof of the Theorem. (See section \ref{s:groups} for definitions).

\section{Permutational wreath products and the piecewise mother groups.}\label{s:groups}

The groups used in our construction will be \textit{permutational wreath products} over the \textit{piecewise mother groups} that were used in \cite{AV}. Even though all properties we need from these groups will follow from results in \cite{AV}, and they could be treated simply as "black boxes", we give a brief description of these groups to make the construction more explicit.

\subsection{Piecewise mother groups.}
Given a bounded sequence $\mathbf{m}=\{m_i\}_{i\geq 1}$  of integers $m_i\geq 2$, we define $\tm$, the spherically symmetric rooted tree with degree sequence $\mathbf{m}$ to be the following graph. The vertex set consists of all finite sequences $a_k\cdots a_1$ where $a_i\in\{0,\ldots,m_i-1\}$, where the empty sequence $o$ is the root. The edge set consists of all pairs of vertices of the form $\{a_k\cdots a_1,\,a_{k-1}\cdots a_1\}$. We set $m^*=\max m_i$.

Let $\Aut(\tm)$ be the set of rooted automorphisms of $\tm$, i.e.  automorphisms fixing the root.
The piecewise mother groups $\mm$ will be subgroups of $\Aut(\tm)$.

Automorphisms $\gamma\in \Aut(\tm)$ can be written as a product
$$
\gamma=\langle \gamma_0,\ldots, \gamma_{m_1-1}\rangle \sigma
$$
where $\sigma\in \text{Sym}(m_1)$ permutes the subtrees of $o$ and $\gamma_i$
are automorphism of the subtrees. Thus the $\gamma_i$ are elements of $\Aut(\ttm)$ where $\ttm$ is the tree with the shifted degree sequence $\tau m:=m_2,m_3,\ldots$. The natural action of the group $\Aut(T_m)$ on infinite strings (the boundary of the tree) $\cdots a_3a_2a_1$ can be defined recursively by $$\cdots a_3a_2a_1.\gamma=(\cdots a_3a_2.\gamma_{a_1}) a_1.\sigma$$

Let $\Pi$ be the group of all automorphisms of the form $p_\mathbf{\sigma} =\langle id,id,\ldots,id \rangle\sigma$ with $\sigma\in \text{Sym}(m_1)$. (These simply permute the children of the root according to $\sigma$.)
Denote by $\tau_\ell$ the cyclic permutation on $\{0,\ldots,m_\ell-1\}$. Define the automorphism $\mathbf{\rho}\in \tm$ recursively
    $$
    \rho_\ell=\langle \rho_{\ell+1},\tau_{\ell+1},\cdots,\tau_{\ell+1}\rangle\,id
    $$
Set $k$ to be the  least common denominator of $\{m_2,m_3,\ldots\}$ (Which is a subset of $\{2,\ldots, m^*\}$).  Let $H$ be the cyclic group $\{\mathbf{\rho}^i\}_{0\leq i<k}$.

We define the \textbf{piecewise mother group} $\mm$ as the group generated by $H$ and $\Pi$.

We will also specify the random walk on $\mm$. We consider the random walk $Y_n$ on $\mm$ where the step distribution is the even mixture of uniform measures on $\Pi$ and $H$.
Note that a uniformly chosen element of $\Pi$ acts on infinite strings by replacing the first (least significant) digit with a uniform number in $0,\dots,m_1-1$.
A uniform element of $H$ acts by replacing the digit after the first non-zero digit with a uniform number in $0,\dots,m_\ell-1$ (where $l$ is the position of the digit replaced). For more on the piecewise mother groups and their action the reader if referred to \cite{AV}.

\subsection{Permutational wreath products.}
Let $\Gamma$ be a finitely generated countable infinite group acting on a set $S$. We single out an element $o\in S$ and call it the root.
Let $\mu$ be a finitely supported symmetric measure on $\Gamma$. Let $\Lambda$ be a finitely generated countable (possibly finite) group. The \textbf{permutational wreath product} $\Lambda \wr_S \Gamma$ is the semidirect product of $\Lambda^S$ with $\Gamma$ acting on it by permuting the coordinates. The multiplication rule, for $\ell,\ell'\in (\Lambda^S)$ and $g,g'\in \Gamma$ is
$$
(\ell,g)(\ell',g')=(\ell\ell'^{g^{-1}},gg')
$$
where $\ell'^{g^{-1}}$ is defined by $\ell'^{g^{-1}}(s)= \ell'(s.g)$.

A \textbf{switch} is a random element of $\Lambda \wr_S \Gamma$ of the form $(\bar L,id_\Gamma)$, with
$$
\bar L(s)=\begin{cases}id_\Lambda \text{ if } s\not=o, \\
L \text{ if } s=o,
\end{cases}
$$
where $L$ is a random element of $\Lambda$ chosen from a fixed symmetric finitely-supported measure.

We consider the random walk
$$
X_n= \prod_{i=1}^n \bar L_i G_i \bar L'_i
$$
(called the \textbf{switch-walk-switch} random walk) on the permutational wreath product. Here the $G_i$ are independent choices from the measure on $\Gamma$, and the $L_i,L_i'
$ are independent choices from the measure on $\Lambda$. We have $X_i=(\mathfrak{L}_i, Y_i)$ where $\mathfrak{L}_i\in \Lambda^S$ and $Y_i=G_1\cdots G_i\in \Gamma$.

\section{Bounds on $\ev|X_n|$ and $H(X_n)$ for $\Lambda \wr_S \mm$.}\label{s:bounds}

In this next sections section we focus on the groups $\Lambda \wr_S \mm$, where $\Lambda$ is infinite, $S$ is taken to be the orbit of the all-zero ray under the natural action of $\mm$, and the root $o$ is taken to be the all-zero ray.
In \cite{AV}, upper and lower bounds were derived for the speed and entropy of the switch-walk-switch random walk on permutational wreath products. When the group action is "significant enough", as happens for $\Lambda \wr_S \mm$, these bounds can be stated in terms of the speed and entropy functions of $\Lambda$ and the return probabilities  of the random walk $o.X_n$ on $S$.

Denote by $T$ the first return time of $o$ under the walk $o.X_n$ on $S$.
Given a group $\Lambda$ and some random walk $R_n=L_1\cdots L_n$  on it, let $$\overline{\lambda}(n)=\max_{k\leq n}\ev|R_k|, \, \ \ \ \underline{\lambda}(n)=\inf_{k\geq n}\ev|R_k| \ \ \ \text{and} \,\,\, h(n)=H(R_n).$$
 We extend the definitions to $\R_{\geq 1}$ by linear interpolation between the integers.
 In the cases we discuss below the choice of random walk on $\Lambda$ does not matter, as long as it is symmetric, finitely supported and it's support generates $\Lambda$, and we therefore leave it unspecified.
 
The majority of the bounds we give are stated (in a slightly different language) in \cite{AV}, while the lower bound on the entropy we will use is a variation on the lower bound on the speed used there. 
Heuristically these bounds can be derived by assuming that the switch moves are the ones making the essential contribution to both speed and entropy of the walk, and that the switch moves may be analyzed as if they are evenly distributed on the support of $\mathfrak{L}_n$.
 
\begin{lemma}\label{l:lower entropy}
The random walk $X_n$ on the permutation wreath product satisfies for $n\ge 1$
$$
H(X_n) \succeq np h(1/p).
$$
for every $p\le \pr(T>n)$.
\end{lemma}
\begin{proof}
The proof follows closely that of Theorem $9$ in \cite{AV} in which the bound $\ev|X_n|\succeq np\underline \lambda(1/p)$ is proved, and we omit the details. One needs only to replace all instances of $\underline{\lambda}$ with $h$ and argue that $H(X_n)\geq \ev\left(\sum_{s\in S}h(\#\text{switch moves at s})\right)$ since the switch moves at the different points at $s$ are independent given the number of switches made at each point. The implicit constants are absolute (do not depend on $\mm$ or $\Lambda$).
\end{proof}

\begin{proposition}\label{p:bounds}
Let $\Lambda$ be some infinite finitely generated group. Let $\textbf{n}=(m_\ell)$ be some bounded degree sequence, and consider the switch-walk switch random walk $X_n$ on $\Lambda\wr \mm$.
For every $p\le \pr(T>n)$ and for every $q\ge \sum_{i=0}^n \pr (T>i)$ we have
\begin{eqnarray*}
\ev |X_n| &\succeq& np\underline \lambda(1/p) \\
\ev |X_n|&\preceq &q\overline \lambda\left(n/q\right)
\\
H(X_n) &\succeq& nph(1/p)
\\
H(X_n) &\preceq&  q\left(h\left(n/q\right) + \log (n+1)\right).
\end{eqnarray*}
With the constants depending on $m^*=\max m_\ell$ only.
\end{proposition}
\begin{proof}
This follows from the bound in Lemma \ref{l:lower entropy} and Corollary $14$ of \cite{AV}, which gives general bounds for permutational wreath products,  together with Corollary $24$ of \cite{AV} which shows that for the case of $\Lambda \wr_S \mm$ the terms appearing in this proposition are the dominant ones.
The constants involved can be chosen to depend only on $m^*=\max m_\ell$, and this dependence only enters through the use of Corollary $24$.
\end{proof}

\section{The extreme cases of the speed-entropy relation.}\label{s:extremes}

To prove Theorem \ref{t:main speed entropy}, we will first construct two groups and random walks on these groups, for which the speed and entropy lie (almost) on the extremes of the speed-entropy relation. That is
one random walk $X_n'$ for which $\ev|X_n'| \asymp H(X_n')$ and one random walk $X_n''$ for which $\ev|X_n''| \succeq \sqrt{\frac{nH(X_n'')}{\log (n+1)}}$. We will then combine the two constructions to prove Theorem \ref{t:main speed entropy}. The "extreme" cases are summed up in the following theorem:

\begin{theorem}\label{t:extremes}
For any $\gamma\in[3/4,1)$ and any function $f:\R_+\rightarrow \R_+$ satisfying
$f(1)=1$ and the log-Lipshitz condition that for all real $a,n\ge 1$
$$a^{3/4}f(n)\leq f(an)\leq a^\gamma f(n)$$
 there is
a bounded sequence $\mathbf{m}=(m_\ell)$ such that
\begin{enumerate}
\item \label{Th:lower entropy} The switch-walk-switch random walk $X_n$ on $\Z \wr_S {\mathcal M_{m}}$ satisfies that
$$
\ev|X_n|  \asymp f(n) , \, \, H(X_n) \asymp \frac{f^2(n)\log (n+1)}{n}.
$$
\item \label{Th:upper entropy}The switch-walk-switch random walk $X_n$ on $(Z_2 \wr \Z) \wr_S {\mathcal M_{m}}$ satisfies that
$$
\ev|X_n|  \asymp H(X_n) \asymp f(n).
$$

\end{enumerate}\
with constants in both clauses depending on $\gamma$ only.
\end{theorem}

To prove the Theorem we will use the following corollary, which combines the bounds in Proposition \ref{p:bounds} with the construction of degree sequences $\mathbf{m}$ for which the first return time $T$ of $o.X_n$ can be sufficiently controlled.

\begin{corollary}\label{c:cor} For any $\gamma\in[1/2,1)$ and any function $g:\R_+\rightarrow \R_+$ satisfying
$g(1)=1$ and the log-Lipshitz condition that for all real $a,n\ge 1$
$$a^{1/2}g(n)\leq g(an)\leq a^\gamma g(n)$$
 there is
a bounded sequence $\mathbf{m}=(m_\ell)$ such that the switch-walk-switch random walk $X_n$ on $\Lambda \wr_S \mm$ satisfies
\begin{eqnarray}
H(X_n) &\succeq&   g(n)h(\frac{n}{g(n)}).
\\
H(X_n) &\preceq&   g(n)\left( h(\frac{n}{g(n)}) + \log(n+1) \right).
\end{eqnarray}
And \begin{eqnarray}
 \ev|X_n| &\succeq& g(n)
 \underline \lambda\left(\frac{n}{g(n)}\right),
 \\
\ev |X_n| &\preceq& g(n) \overline \lambda\left(\frac{n}{g(n)}\right).
\end{eqnarray}
With constants depending on $\gamma$ only.
\end{corollary}
\begin{proof}
Corollary $19$ and Lemma $20$ of \cite{AV} state that one can choose the degree sequence $\textbf{m}=(m_\ell)$ so that $\pr(T>n)\succeq \frac{g(n)}{n}$ and $\sum_{i=0}^n \pr (T>i)\preceq g(n)$,
with constants depending only on $\gamma$ (through the choice of $m^*$). The conclusion now follows from Proposition \ref{p:bounds}.
\end{proof}

\begin{proof}[Proof of Theorem \ref{t:extremes}]
To prove the theorem, we consider the group $\Lambda \wr_S {\mathcal M_{m}}$  and describe several cases for the choice of the group $\Lambda$, using the above Corollary to bound $\ev|X_n|$ and $H(X_n)$.
Define $g(n) = \frac{f^2(n)}{n}$.
Note that a function $f$ satisfies the conditions of the Theorem if and only if $g$ satisfies the conditions of Corollary \ref{c:cor}.
\begin{enumerate}
\item To get the first clause we take $\Lambda = \Z$, which is the choice used in \cite{AV} to construct groups with prescribed speed behaviour. Then $\overline \lambda(n) \asymp \underline \lambda(n) \asymp \sqrt{n}$ and $h(n)\asymp \log (n+1)$. We thus get
$$\ev|X_n| \asymp \sqrt{ng(n)} \asymp f(n) , \,\,\ \   H(X_n)\asymp g(n)\log (n+1).$$

\item To get the second clause we take $\Lambda = \Z_2 \wr \Z$. We then have $\overline \lambda(n) \asymp \underline \lambda(n) \asymp \sqrt{n}$ and $h(n)\asymp \sqrt{n}$. We thus get
$$\ev|X_n|  \asymp f(n) , \,\,\ \   H(X_n)\asymp  f(n).$$

\end{enumerate}
\end{proof}
Note that if we took $\Lambda=\Z_2$ (or any other finite group), we would get (using similar bounds proved in \cite{AV})
$$ g(n) \preceq \ev|X_n| \preceq  \sqrt{ng(n)} , \,\ \  H(X_n)\asymp g(n).$$
This was used to get precise entropy behaviour in \cite{AV} but does not give any good control of the speed (at least via known methods).

\section{Proof of Theorem \ref{t:main speed entropy}.}\label{s:proof}
\begin{proof}
By Theorem \ref{t:extremes} clause $(1)$ there is a sequence $\mathbf{m'}=(m'_\ell)$ so that the switch-walk-switch random walk $X_n'$ on $\Gamma' = \Z \wr_{S'} {\mathcal M_{m'}}$ satisfies
$$ \ev|X_n'| \asymp f(n) , \,\, H(X_n') \asymp \frac{f^2(n)}{n}\log (n+1)$$
By Theorem \ref{t:extremes} clause $(2)$ there is a sequence $\mathbf{m''}=(m''_\ell)$ so that the switch-walk-switch random walk $X_n''$ on $\Gamma'' = \Z \wr_{S''} {\mathcal M_{m''}}$ satisfies
$$ \ev|X_n''| \asymp h(n), \,\, H(X_n'')\asymp h(n)$$
Note that  the relation $h(n)\preceq f(n)\preceq \sqrt{\frac{nh(n)}{\log (n+1)}}$ implies that
$$ \max(\ev|X_n'|,\ev|X_n''|) \asymp f(n), \,\, \max(H(X_n',X_n'')) \asymp h(n).$$
Define $\Gamma = \Gamma' \times \Gamma''$ and let $X_n=(X_n',X_n'')$ be a random walk on $\Gamma$ defined by making an independent step of the switch-walk-switch random walk in each coordinate every step.
We claim $\Gamma,X_n$ satisfy the requirements of the Theorem.
Indeed, both speed and entropy clearly satisfy
$$\max(\ev|X_n'|, \ev|X_n''|)\leq \ev|X_n|\leq \ev|X_n'| + \ev|X_n''|$$
$$\max(H(X_n'), H(X_n''))\leq H(X_n)\leq H(X_n') + H(X_n'').$$
Therefore
$$\ev|X_n| \asymp \max(\ev|X_n'|,\ev|X_n''|) \asymp f(n)$$
and $$ H(X_n) \asymp \max(H(X_n'), H(X_n'')) \asymp h(n).$$
With all constants depending on $\gamma$ only.
\end{proof}

\section{Conjectures and open questions.}\label{s:open}
As in the previous parts of this paper, by a random walk on a group we mean a random walk with a symmetric finitely supported step measure on an infinite finitely supported group.
In order to make the conjectures and questions more concise, we will focus on speed and entropy functions of the form $n^{\beta+o(1)}$ and $n^{\delta+o(1)}$ respectively. In this case, $\beta$ and $\delta$ are called the speed and entropy exponents of the walk.

In \cite{AV} it was conjectured that all speed exponents between $\frac{1}{2}$ and $1$ are attainable.
We would like to further generalize this conjecture to claim that when both the speed and entropy exponents are above $\frac{1}{2}$, essentially the only constraint on the joint behaviour is the speed-entropy relation:
\begin{conjecture}
For any $\frac{1}{2}\leq \beta\leq 1$ and any $\frac{1}{2} \leq \delta\leq 1$ satisfying $\delta\leq \beta\leq 2\delta-1$ there exists a group $G$ and a (symmetric finitely supported) random walk $X_n$ on $G$ satisfying $\ev|X_n|\asymp n^\beta$ and $H(X_n)\asymp^* n^\delta$, where $\asymp^*$ denotes equality up to polylog factors.
\end{conjecture}

When the entropy exponent drops below $\frac{1}{2}$, things are less clear. It is conjectured that one must move to the realm of sub-exponential growth, where even examples of speed exponents above $\frac{1}{2}$ are not known. Relaxing the precision required,  we therefore ask:
\begin{question}
For what values of $0\leq \delta \leq \frac{1}{2}$ and $\frac{1}{2}\leq \beta \leq 1$  does there exist a random walk $X_n$ on a group $G$ such that $ \ev|X_n|\asymp n^{\beta+o(1)}$ and $H(X_n)\asymp n^{\delta+o(1)}$ ?
\end{question}

The groups we considered are permutational wreath products by piecewise mother groups. It is still an open question to find the speed or entropy behaviour for the piecewise mother groups themselves, even in the case of when the degree sequence $\mathbf{m}$ is constant. In fact, this is open on any of the classical automaton groups (except perhaps the lamplighter group); we will skip the definition of these groups but they can be found in, for example \cite{AAV}.
\begin{problem} [\cite{AV} Problem $32$]
Find the asymptotic behaviour  the speed exponent or the entropy exponent in any of the following groups: Grigorchuk's group, Basilica group,  Hanoi towers group, piecewise mother groups.
\end{problem}

As mentioned in the introduction, beside the rate of escape and the entropy there are several other quantities related to groups and random walks on groups that have been extensively studied and have many connections with geometric and analytic properties of the groups involved. Among the most notable examples are the volume growth of the group, the return probabilities of the random walk and the Hilbert compression exponent of the group (See \cite{NP} for definition). The question studied in this paper can be seen as a case of the more general question regarding the possible joint behaviour of these quantities.
One possible formulation is as follows.
\begin{question}\label{q:general}
For which tuples $(\alpha,\beta,\gamma,\delta,\rho)$ does there exist a group $G$ with Hilbert compression exponent $\alpha$, volume growth $e^{n^{\rho+o(1)}}$ and a random walk  $X_n$ on that group with $$\ev|X_n|=n^{\beta+o(1)},\,\,\, P(X_n=id)=e^{-n^{\gamma+o(1)}},\,\,\, H(X_n)=n^{\delta+o(1)} ?$$
\end{question}
There are many known relations between the different quantities (See e.g. the references in section \ref{s:intro}), but very few constructions allowing to control more than one of them simultaniously. For more on the various exponents and their relations we refer the reader to \cite{Gou}.

\smallskip
\noindent{\bf Acknowledgments.}
 We wish to thank Laurant Bartholdi for a discussion concerning question \ref{q:general} and Miklos Abert, Omer Angel, Balint Virag and the BIRS centre for the wonderful conference on Groups, Graphs and Stochastic Processes during which this work was initiated. The author's work was supported by the Israel Science Foundation grant ISF $1471/11$ and by a young GIF grant $\# I-1121-304.1-2012$ from the German Israeli foundation.
\bibliographystyle{alpha}

\bigskip\bigskip\bigskip\noindent
\begin{minipage}{0.49\linewidth}
Gideon Amir
\\Department of Mathematics
\\Bar Ilan University
\\Ramat Gan, 52900 Israel
\\{\tt amirgi@macs.biu.ac.il }
\end{minipage}

\end{document}